\documentclass[12pt]{amsart}
\usepackage{geometry}                % See geometry.pdf to learn the layout options. There are lots.
\usepackage [pagewise, modulo, displaymath, mathlines] {lineno}
\usepackage[parfill]{parskip}    % Activate to begin paragraphs with an empty line rather than an indent
\usepackage{graphicx}
\usepackage{float}
\usepackage[caption = false]{subfig}
\usepackage{amssymb}
\usepackage{epstopdf}
\usepackage{amsmath,amssymb,latexsym,pstricks,mathrsfs,comment,amsthm,graphicx,tikz,enumerate,pgffor,cite,wrapfig,float}

\newtheorem{theorem}{Theorem}

\newtheorem{lemma}[theorem]{Lemma}

\newcommand{\nc}{\newcommand}
\nc{\rnc}{\renewcommand}
\nc{\ip}{idempotent}
\nc{\Ip}{Idempotent}
\nc{\sgp}{semigroup}
\nc{\alg}{algebra}
\nc{\sla}{semilattice}
\nc{\ioi}{if and only if~}
\nc{\cs}{completely simple}
\nc{\ch}{character}
\nc{\supp}{\operatorname{supp}}

\title{Groupoids on a skew lattice of objects}
\author{d. g. fitzgerald}
\address{School of Mathematics and Physics, University of Tasmania, Private Bag 37, Hobart 7001, Australia }
\email{D.FitzGerald@utas.edu.au}
\begin{document}
\dedicatory{To Jonathan Leech---Na zdravje!}
\maketitle

\begin{abstract}

Motivated by some alternatives to the classical logical model of boolean algebra, this paper deals with algebraic structures which extend skew lattices by locally invertible elements. 
Following the meme of the Ehresmann-Schein-Nambooripad  theorem, we consider a  groupoid (small category of isomorphisms)  in which the set of objects carries the structure of a skew lattice.  The objects act on the morphisms by left and right restriction and extension mappings of the morphisms, imitating those of an inductive groupoid.  
Conditions are placed on the actions, from which pseudoproducts may be defined.  This gives an algebra of signature $(2,2,1)$, in which each binary operation has the structure of an orthodox semigroup.  In the reverse direction, a groupoid of the kind described may be reconstructed from the algebra.%for each band operation (a non-commutative ``meet'' and ``join''), and links between them arising from the skew lattice property of the common set of idempotents (the groupoid objects or identities).

\end{abstract}

\section{Non-commutative and non-idempotent lattice analogues}\label{intro}
As non-classical logics have been developed for various knowledge domains, so various algebras have been proposed as extensions or alternatives to the classical model of boolean algebra. A significant one for this paper is the theory of skew lattices; for a contemporary account, see Leech's surveys \cite{Le89, Le96}.  We provide the details of our notation in Section \ref{grpds2}.  Another proposal is that of MV-algebras, and their co\"{o}rdinatisation via inverse semigroups as described by Lawson and Scott~\cite{LS}.
%partition algebra of D. Ellerman~\cite{El1, El2, El3}, MV-algebras~\cite{LS},~\cite{Le}. 
Thus one theme is to allow  
sequential operations and hence non-commutative logical connectives, and another introduces non-idempotent connectives. 
This note will consider a combination of these themes, by seeking reasonable structures which extend skew lattices by locally invertible elements. 
%After a brief and partial survey of these selected proposals, the talk 
%A couple of construction methods will be given, resulting in what we may think of as orthodox double semigroups.  The simpler of these is the semidirect product of a group and a band, considered already by Hartmann and Szendrei \cite{Ha, HSz} and quite recently by Zenab \cite{Ze}, which we extend to the semidirect product of a group and a skew lattice.  
 
The principal tool in this construction is based on the ideas behind the Ehresmann-Schein-Nambooripad (ESN) theorem, of which a full account is given in Chapter 4 of Lawson's book \cite{La}: we consider a small category of isomorphisms  in which the set of objects carries the structure of a skew lattice.  We postulate that the objects act (partially) on the morphisms by left and right restriction and extension mappings of the morphisms, imitating those of an inductive groupoid.  
Certain reasonable conditions are postulated, and from these a suitable pseudoproduct is defined, much as in the inverse semigroup case, for each skew lattice operation (a non-commutative ``meet'' and ``join'').  This results in a total algebra involving two orthodox semigroups with a common set of idempotents isomorphic to the given skew lattice. % under induced operations.

%and links between the two structures. arising from the skew lattice property of the common set of idempotents (the groupoid objects or identities).  
Because of the complexity involved in having two operations, we begin by considering a groupoid over a set of objects with a single band operation.  A much more general situation has been studied, under the name of \emph{weakly $B$-orthodox} semigroups, by Gould and Wang in \cite{GW}, but because the present author has been unable to find this special case treated in the literature, a detailed account will be given here.  %For the present, we consider $B$ as a lower band only, postponing use of the join operation to the next section.
Later sections deal with the pair of linked band operations, construct the total algebra described above, and show how the original groupoid may be recovered from the algebra.  Aspects of the constructions which need further elaboration are noted in the final section.

\section{Groupoids on a band of objects} \label{grpds1}
Let us recall from \cite{La} that an inverse semigroup is equivalent to an \emph{inductive groupoid}, i.e., 

\begin{itemize}
\item a (small) category of isomorphisms with 
\item a meet operation on objects and 
\item a notion of restriction of a morphism to any of its subdomains.
 
\end{itemize}

We attempt something similar here, but changing the conditions on the set of objects.  Let $\mathscr{G}$ be a groupoid with composition $\circ$ and $B$ its set of objects, endowed with an associative and idempotent operation $\wedge$. %satisfying the absorptive axioms 
%\[a\vee (a\wedge b) = a = a\wedge (a\vee b)~,\hspace {3em} (a\wedge b)\vee b  = b =  (a\vee b)\wedge b\] for a skew lattice  \cite{Le}.  
Then  $(B, \wedge) $ is known as a \emph{lower} band, and possesses
%, that is, idempotent semigroups.  Moreover each has 
a pair of natural preorders:  we write
 
\begin{itemize}
 \item $a\leq_{L} b$ \ioi $a = a\wedge b$, and  $a\leq_{R} b$ \ioi $a = b\wedge a$.
\end{itemize}
%
%and in the \emph{upper} band $(B,\vee) $ we write 
%\begin {itemize}
%\item $a\geq_{L} b$ \ioi $a = a\vee b$, and $a\geq_{R} b$ \ioi $a = b\vee a$. 
%\end{itemize}  
%We do not at this stage admit the usual convention that $\leq$ and $\geq$ are converse relations! 
%The skew lattice law implies that $ a=a\vee b \iff a\wedge b =b \text{   and   } a\vee b = b \iff a = a\wedge b$, so that $a \leq_{L} b$ \ioi  $a=a\wedge b$ \ioi $b \geq_{R} a$; which is to say that $ \leq_{L}$ and $\geq_{R}$ are converse relations, as also 
%$ \leq_{R}$ and $\geq_{L}$.  Thus we need use only two of these four relations, which we mostly choose to be $\leq_{L}$ and $\leq_{R}$.  For the present, we consider $B$ as a lower band only, postponing use of the join operation to the next section.
%(Note this frees us to use $\geq_{L}$ and $\geq_{R}$  in their more traditional senses as inverses to $\leq_{L}$ and $\leq_{R}$ respectively; and we shall use this freedom, without further mention, in the few places where it seems convenient.)

As usual, we may identify each object $b$ with its identity ${\bf i}_{b}$, and write ${\bf d}g$  and ${\bf r}g$ for the domain and range maps in $\mathscr{G}$, thus: ${\bf d} g= g\circ g^{-1}, \: {\bf r}g = g^{-1}\circ g$.  Suppose too that for each $a\in B$ there are \emph{left} and \emph{right restriction} (partial) operations $_{a}|,\, |_{a} :  \mathscr{G} \rightarrow \mathscr{G} $ such that:
\begin{itemize}
\item  $_{a}|g$ is defined whenever $a \leq_{L} {\bf{d}}g $,  %${\bf r}(_{a}| g)$ is written $a^{g}$ (so there are two actions here) 
with 
\item $_{a}|g : a \rightarrow {\bf r}(_{a}| g)\leq_{L}{\bf r}g$;   
\end{itemize}

and (lateral-) dually, 
\begin{itemize}

\item $g|_{a}$ is defined whenever $a \leq_{R} {\bf{r}}g $, with $g|_{a} : {\bf d}(g|_{a}) \rightarrow a$, \: ${\bf d}(g|_{a}) \leq_{L}{\bf d}g$.

\end{itemize} 

%\begin{center}
\begin{figure}[h!]
\caption{Left and right restriction operators}
\label{fig.1}
\subfloat{ 
 \begin{tikzpicture}    %top left
\tikzset{node distance=2.5cm, auto}
  \node (c) {$c$};   \node (x) [below of=c] {$$};
  \node (d) [right of=c] {$d$};
  \draw[->] (c) to node {$g$} (d);
  \node (a) [below of=c] {$a$};
  \node (y) [right of=a]  {${\bf r}(_{a}|g)$}; %a^{g}
  \draw[->] (a) to node [swap]{$_{a}|g$} (y);
  %\node (P1) [node distance=1.4cm, left of=P, above of=P] {$\hat{P}$};
 % \draw[->, bend left] (A.30) to node {$m$} (X.160);
  \draw[red] (c.270) to node [swap]{$a\leq_{L}c$} (a.90);
  %\draw[->,bend right] (B.60) to node [swap] {$n'$} (X.300);
  %\draw[->] (I) to node [swap] {$q$} (B);
  %\draw[->] (I) to node {$p$} (A);
  %\draw[->] (B.120) to node [swap] {$n$} (X.240);
\end{tikzpicture}
}
\hspace{2cm}
\subfloat{   %topright
\begin{tikzpicture}
\tikzset{node distance=2.5cm, auto}
  \node (c) {$c$};   \node (x) [below of=c] {$$};
  \node (d) [right of=c] {$d$};
  \draw[->] (c) to node {$g$} (d);
  \node (a) [below of=c] {${\bf d}(g|_{a})$};
  \node (y) [right of=a]  {$a$}; %a^{g}
  \draw[->] (a) to node [swap]{$g\vert_{a}$} (y);
  %\node (P1) [node distance=1.4cm, left of=P, above of=P] {$\hat{P}$};
 % \draw[->, bend left] (A.30) to node {$m$} (X.160);
  \draw[blue] (d.270) to node {$a \leq_{R}d$} (y.90);
  %\draw[->,bend right] (B.60) to node [swap] {$n'$} (X.300);
  %\draw[->] (I) to node [swap] {$q$} (B);
  %\draw[->] (I) to node {$p$} (A);
  %\draw[->] (B.120) to node [swap] {$n$} (X.240);
\end{tikzpicture}
%\end{center}
}
%\vspace{3mm}
%
%%\begin{center}
%%\begin{subfigures}
%\subfloat{ 
% \begin{tikzpicture}	%bottom left
%\tikzset{node distance=2.5cm, auto}
%  \node (c) {$a$};   
%  \node (x) [below of=c] {$c$};
%  \node (d) [right of=c] {${\bf r}(^{a}\vert g)$};
%  \draw[->] (c) to node {$^{a}\vert g$} (d);
%  \node (a) [below of=c] {$c$};
%  \node (y) [right of=a]  {$d$}; %a^{g}
%  \draw[->] (a) to node [swap]{$g$} (y);
%  %\node (P1) [node distance=1.4cm, left of=P, above of=P] {$\hat{P}$};
% % \draw[->, bend left] (A.30) to node {$m$} (X.160);
%  \draw[blue] (c.270) to node [swap]{$a\geq_{L}c$} (a.90);
%  %\draw[->,bend right] (B.60) to node [swap] {$n'$} (X.300);
%  %\draw[->] (I) to node [swap] {$q$} (B);
%  %\draw[->] (I) to node {$p$} (A);
%  %\draw[->] (B.120) to node [swap] {$n$} (X.240);
%\end{tikzpicture}
%}
%\hspace{2cm}
%%\end{center}
%\subfloat{
%%\begin{center}  
% \begin{tikzpicture}	%bottom right
% \tikzset{node distance=2.5cm, auto}
%  \node (c) {${\bf d}(^{a}|g)$};
%  \node (d) [right of=c] {$a$};
%  \draw[->] (c) to node {$g\vert^{a}$} (d);
%  \node (a) [below of=c] {$c$};
%  \node (y) [right of=a]  {$d$}; %a^{g}
%  \draw[->] (a) to node [swap]{$g$} (y);
%  %\node (P1) [node distance=1.4cm, left of=P, above of=P] {$\hat{P}$};
% % \draw[->, bend left] (A.30) to node {$m$} (X.160);
%  \draw[red] (d.270) to node {$a\geq_{R}d$} (y.90);
%  %\draw[->,bend right] (B.60) to node [swap] {$n'$} (X.300);
%  %\draw[->] (I) to node [swap] {$q$} (B);
%  %\draw[->] (I) to node {$p$} (A);
%  %\draw[->] (B.120) to node [swap] {$n$} (X.240);
% \end{tikzpicture} }
 
\end{figure}
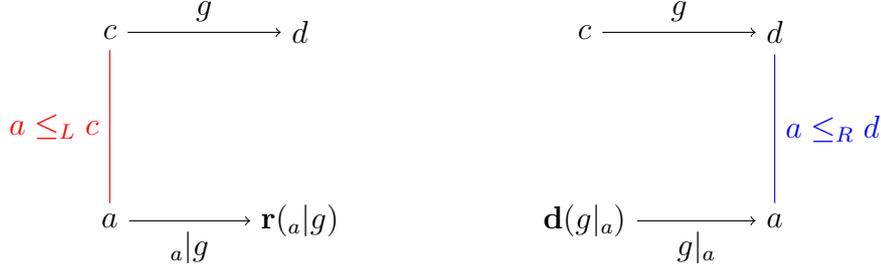
%\end{center}

%The upper row of 
Figure \ref{fig.1} shows the left and right restrictions.  (There are analogous (in fact, vertically dual) requirements for  extension operators, which will %remain implicit for the time being, and are to 
be dealt with more explicitly in Section \ref{grpds2}.)
Certain sensible axioms must be satisfied:

%
%\begin{center}
%\begin{tikzpicture}
%\tikzset{node distance=2.5cm, auto}
%  \node (c) {$c$};   \node (x) [below of=c] {$$};
%  \node (d) [right of=c] {$d$};
%  \draw[->] (c) to node {$g$} (d);
%  \node (a) [below of=c] {$a$};
%  \node (y) [right of=a]  {${\bf r}(_{a}|g)$}; %a^{g}
%  \draw[->] (a) to node [swap]{$_{a}|g$} (y);
%  %\node (P1) [node distance=1.4cm, left of=P, above of=P] {$\hat{P}$};
% % \draw[->, bend left] (A.30) to node {$m$} (X.160);
%  \draw[red] (c.270) to node [swap]{$\leq_{L}$} (a.90);
%  %\draw[->,bend right] (B.60) to node [swap] {$n'$} (X.300);
%  %\draw[->] (I) to node [swap] {$q$} (B);
%  %\draw[->] (I) to node {$p$} (A);
%  %\draw[->] (B.120) to node [swap] {$n$} (X.240);
%\end{tikzpicture}
%\end{center}
%

\begin{itemize}
 \item (identities) \hspace{1em}% or idempotency 
 $_{{\bf d}(g) }|g = g$;
 \item (preorders) \hspace{1em}if $a \leq _{L}b$, then $ _{a}|{\bf i}_{b} = {\bf i}_{a}$;
 \item (transitivity) \hspace{1em} if $a \leq_{L} b \leq_{L}{\bf{d}}g $,
 then $_{a}|g = _{a\wedge b}|g = \,_{a}\!|(_{b}|g)$;  
%\end{itemize}
%  As to composition, we require (with all its duals)
%\begin{itemize}
 \item (composition) \hspace{1em} if $f\circ g$ is defined (so that ${\bf r}\!f = {\bf d}g$), then 
    \[_{a}|(f\circ g) = \,(_{a}|f)\circ (_{{\bf r}(_{a}|f )}|g),\]
    the right-hand composite being defined because ${\bf r}(_{a}|f )\leq_{L}{\bf r}\!f = {\bf d}g$; see Fig. 2.
\end{itemize}

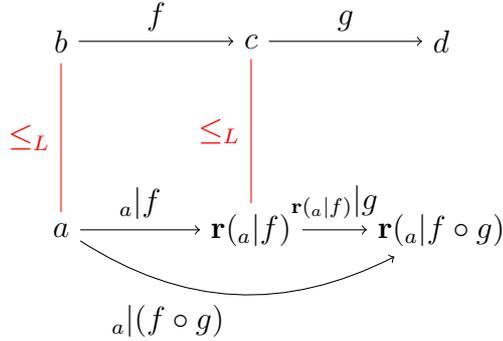
\begin{figure}[h!]
\caption{Restriction of a composite morphism}
\label{fig.2}
%\begin{center}
\begin{tikzpicture}
\tikzset{node distance=2.5cm, auto}
  \node (c) {$c$};   \node (x) [below of=c] {$$};
  \node (d) [right of=c] {$d$};
  \draw[->] (c) to node {$g$} (d);
  \node (z) [below of=c] {${\bf r}(_{a}|f) $};
  \node (y) [right of=z]  {${\bf r}(_{a}\vert f\circ g)$}; %
  \draw[->] (z) to node {$_{{\bf r}(_{a}|f)} \vert g$} (y);
  \draw[red] (c.270) to node [swap]{$\leq_{L}$} (z.90);%should be (z)?
 \node (b) [left of=c] {$b$};
 \node (a) [below of=b] {$a$};
  \draw[red] (b.270) to node [swap]{$\leq_{L}$} (a.90);
  \draw[->] (b) to node {$f$}(c);
  \draw[->] (a) to node {$_{a}|f$} (z);
  \draw[->, bend right] (a) to node [swap] {$_{a}|(f\circ g)$}(y);  
%\node (w) [node distance=1cm,  below of=d] {$a^{f}={\bf r}(_{a}|f )\hspace{6pt}.$};
\end{tikzpicture}
%\end{center}
\end{figure}
\subsection*
 {Actions and conjugacy}
Let us write, without prejudice, $a^{g}$ as an alternative for ${\bf r}(_{a}|g)$, and $^{g}a$ for ${\bf d}(g|_{a})$.  This lightens the notation, and emphasises the similarity to actions and conjugates.  {\it Caution:} However, it should not be taken to mean that anything like $a^{f} = (f^{-1}\vert_{a}) \circ (_{a}\vert f)$, or  
$a^{f} = \,^{f^{-1}}\!a$, or $_{a}\vert f = f\vert_{a^{f}}$ necessarily hold%, or that $_{a}\vert f \neq f\vert_{a^{f}}$, unless $a^{f}\leq_{R}\mathbf{r}\!f$
: in general, $f\vert_{a^{f}} = f\vert_{\mathbf{r}\!f\wedge a^{f}}$.

What we do have, following from $_{a}|(f\circ g) = \,(_{a}|f)\circ (_{{\bf r}(_{a}|f )}|g)$, is that  
\begin{multline*}
_{a}|{\bf i}_{b} = {\bf i}_{a}  =  \,(_{a}|f)\circ (_{a^{f}}|f^{-1}), \text{       so   } \\
 (_{a}|f)^{-1} = \,_{a^{f}}|f^{-1} , \hspace{12pt} a^{{\bf i}_{b}} = a, \hspace{12pt}  \text{   and    } \hspace{12pt} a^{f} = \,(_{a^{f}}\vert f^{-1})\circ (\,_{a}\vert f) ; 
\end{multline*} 
moreover, $a^{f\circ g} = (a^{f})^{g}$.
% \[(_{a}|f)^{-1} = \;_{a^{f}}\!\vert (f^{-1}), \] 
%a condition linking the ``actions''.  Note also that 
%$_{a}\vert f \circ a^{f} = \,_{a}\vert f$.
%The lateral duals (dealing with right restrictions) must also hold.  
%(Note that these are maps that look rather like (partial) actions.  
%Treat them with caution:  $_{a}\vert f \neq f\vert_{a^{f}}$; and $ ^{f^{-1}}\!a \neq a^{f}$ (so $a^{f} \neq f^{-1}\wedge a\wedge f$).
%But we do have \[ a^{f} = a^{f}\wedge f^{-1} \wedge a\wedge f = \dots\]. )
We want to link right and left ``actions''  by 
%\begin{itemize}
 $ (_{a}\vert f)\vert_{b} ~=~ _{a}\!\vert (f\vert_{b}), $ but there is a little problem here, since one of the two sides of that equation may fail to be defined while the other is defined.  We therefore seek  to extend the conditionally-defined restrictions to total maps by the following device, based on the pseudoproduct construction familiar from the ESN theorem:

For $g\in \mathscr{G}$ and any $a\in B$, define  $_{a}|g := _{a\wedge {\bf d}g}\!\vert g$, the right hand side being meaningful since $a\wedge{\bf d}g \leq_{L}{\bf d}g $.  (Note that if $a\leq{\bf d}g$ already, the notations agree).  The next figure shows the situation (where, also by extension, we write $a^{g}$ for the already-defined  $(a\wedge {\bf d}g)^{g}$).

\begin{figure}[h!]
\caption{Generalised restriction of $g$ to object $a$ and action of $g$ on $a$}
\label{fig.3}
\begin{tikzpicture}
\tikzset{node distance=2.5cm, auto}
  \node (b) {$b$};   \node (x) [below of=b] {$$};
  \node (c) [right of=b] {$c$};
  \node (a^b) [node distance=1cm,  left of=x] {$a\wedge b$};
  \node (y) [node distance=3cm, right of=a^b]  {$(a\wedge b)^{g}=a^{g}$};
  \draw[->] (a^b) to node [swap]{$_{a}|g$} (y);
  \draw[->] (b) to node {$g$} (c);
  \draw[red] (b) to node {$\leq_{L}$} (a^b);
  %lefthand factor
  \node (a) [node distance=2cm,  left of=b]{$a$};
  %\node(z) [left of=a]{$z$};
  %\node(e)[left of=a^b]{$^{f}(a\wedge b)$};
  %\draw [->](z)to node {$f$}(a);
  \draw [blue] (a) to node [swap]{$$}(a^b);
  %\draw[->] (e) to node[swap] {$f\vert_{a\wedge b}$} (a^b);
  %\draw[->] (I) to node [swap] {$q$} (B);
  %\draw[->] (I) to node {$p$} (A);
  %\draw[->] (B.120) to node [swap] {$n$} (X.240);
\end{tikzpicture}
\end{figure}
 Then if $g={\bf i}_{b}$, we have $a\wedge {\bf i}_{b} = _{a\wedge b}\!\vert{\bf i}_{b} = {\bf i}_{a\wedge b} = 
 {\bf i}_{a} \wedge {\bf i}_{b} = {\bf i}_{a}\wedge b = {\bf i}_{a}\vert_{a\wedge b} $, and we may write $a\wedge g$ for $_{a}\vert g$ without conflict.  A little re-writing of definitions shows that 
 \begin{equation} \label{abg}
 (a\wedge b)\wedge g = _{a\wedge b}\vert g = \;_{a}\vert(_{b}\vert g) = a\wedge (b\wedge g) 
 \end{equation}
 and 
 \begin{equation} \label{agg*}
 a\wedge \mathbf{d}g = \mathbf{d}(a\wedge g) = (a\wedge g)\circ (a\wedge g)^{-1}.
 \end{equation}
We  complete our list of postulates with the previously-mentioned $ (_{a}\vert f)\vert_{b} ~=~ _{a}\!\vert (f\vert_{b}), $ which we now write as 
\begin{itemize} 
  \item $(a\wedge f)\wedge b = a\wedge (f\wedge b) ,$ for all $a,b\in B$ and $f\in \mathscr{G}$ .
\end{itemize}
(More fully, this is $ (_{a\wedge {\mathbf d}f}\vert f)\vert_{a^{f}\wedge b} ~=~ _{a\wedge \!^{f}b}\!\vert (f\vert_{b\wedge \mathbf{r}f}) $ .)

Next, we may extend the composition further, to a \emph{pseudoproduct} $\otimes$: when $f:z\rightarrow a$ and $g:b\rightarrow c$, we define
\[f\otimes g := (f\vert_{a\wedge b})\circ (_{a\wedge b}\vert g)=(f\wedge (a\wedge b))\circ ((a\wedge b)\wedge g) . \]

\begin{figure}[h!]
\caption{Diagram illustrating the pseudoproduct}
\label{fig.4}
% \begin{center}
\begin{tikzpicture}
\tikzset{node distance=2.5cm, auto}
  \node (b) {$b$};   \node (x) [below of=b] {$$};
  \node (c) [right of=b] {$c$};
  \node (a^b) [node distance=1cm,  left of=x] {$a\wedge b$};
  \node (y) [node distance=3cm, right of=a^b]  {$(a\wedge b)^{g}$};
  \draw[->] (a^b) to node [swap]{$_{a\wedge b}\vert g$} (y);
  \draw[->] (b) to node {$g$} (c);
  \draw[red] (b) to node {$\leq_{L}$} (a^b);
  %lefthand factor
  \node (a) [node distance=2cm,  left of=b]{$a$};
  \node(z) [left of=a]{$z$};
  \node(e)[left of=a^b]{$^{f}(a\wedge b)$};
  \draw [->](z)to node {$f$}(a);
  \draw [blue] (a) to node [swap]{$\leq_{R}$}(a^b);
  \draw[->] (e) to node[swap] {$f\vert_{a\wedge b}$} (a^b);
  %\draw[->] (I) to node [swap] {$q$} (B);
  %\draw[->] (I) to node {$p$} (A);
  %\draw[->] (B.120) to node [swap] {$n$} (X.240);
\end{tikzpicture}
\end{figure}
Then $a\wedge f$ is actually just ${\bf i}_{a}\otimes f$.  This is indeed an extension of meaning: when $f\circ g$ is defined, $f\otimes g = f\circ g$, and when $f = {\bf i}_{a}$ and $g = {\bf i}_{b}$, 
\[f\otimes g = {\bf i}_{a}\otimes{\bf i}_{b} = {\bf i}_{a\wedge b} = {\bf i}_{a}\wedge {\bf i}_{b} ;\]
so we may as well use just the one symbol $\wedge$ for $\otimes$, as it extends $\circ$ and the restrictions, as well as the original $\wedge$ on $B$.  Let us check remaining non-trivial cases for associativity.

\begin{lemma}\label{1} 
 For all $f, g \in \mathscr{G}$ and $e\in B$, with $f\colon \mathbf{d}f \rightarrow a$ and $g\colon b \rightarrow \mathbf{r}g$, 
 \newline \emph{ (i)}  $ (f\wedge e)\wedge g = f\wedge (e\wedge g) $,
  \newline \emph{ (ii)} $e^{f\wedge g} = (e^{f})^{g}$, and 
 \newline \emph{ (iii)} $ _{e}\vert (f\wedge g) = (_{e}\vert f)\wedge g$.
 %\newline \emph{ (i)}$ (f\wedge e)\wedge g = f\wedge (e\wedge g) $.
 \end{lemma}
 \begin{proof}
 (i)  By definition, $ (f\wedge e)\wedge g = (f\vert_{e})\wedge g = (f\vert_{a\wedge e})\vert_{b}\;\circ \;_{a\wedge e}\vert g = (f\vert_{a\wedge e\wedge b})\;\circ \;(_{a\wedge e\wedge b}\vert g)$, while $ f\wedge (e\wedge g) = f\vert_{e\wedge b} \wedge \:_{e\wedge b}\vert g = f\wedge (e\wedge g) $.%(f\vert_{a\wedge e\wedge b})\;\circ \;_{a\wedge e\wedge b}\vert g$
 %f\wedge (e\wedge g)

 (ii)  We already have $e^{f\wedge g} =  e^{(f\vert_{b})\circ(_{a}\vert g)} =  (e^{f\vert_{b}})^{(_{a}\vert g)}.  $
%With the aid of the following diagram, which is enlarged from Figure \ref{fig.4} above using abbreviated notation, we show that $e^{f\wedge g} = (e^{f})^{g}$.  
Observe that  $e\wedge \mathbf{d}(f\vert_{b})\leq_{R}  e\wedge \mathbf{d}f$, since $\leq_{R}$ is left compatible (Fig. \ref{assist} %, which is enlarged from Figure \ref{fig.4} above using abbreviated notation,
 may assist the reader).  So $e\wedge f = e\wedge f\vert_{b}$  and $e^{f\vert_{b}} = e^{f}$.  Likewise $e^{f}\wedge b\leq_{L} a\wedge b$ and $(e^{f})^{g} = (e^{f})^{(_{a}\vert g)} = e^{f\wedge g}.$

(iii) Using $e^{f\vert_{b}} = e^{f}$ from (ii), we have 
\begin{align*}
   _{e}\vert (f\wedge g) &= \;_{e}\vert (f\vert _{b}\:\circ\; _{a} \vert g) &=& &
 \;(_{e}\vert (f\vert _{b}))\;\circ\; (_{e^{f\vert _{b}}} \vert (_{a} \vert g)) \\
&= \;((_{e}\vert f)\vert _{b})\;\circ\; (_{e^{f\vert _{b}}\wedge{a}} \vert g) &=& &\;((_{e}\vert f)\vert _{b})\;\circ\; (_{e^{f}} \vert g) \\
&=(_{e}\vert f)\;\wedge\;  g .
 \end{align*}
%(_{e}\vert f)\wedge g (_{e}\vert f)\wedge g$
\end{proof} 
\begin{figure}[!h]
\caption{Diagram illustrating  $e^{f\wedge g} = (e^{f})^{g}$}
\label{assist}
\begin{tikzpicture}
\tikzset{node distance=2.2cm, auto}
  \node (b) {$b$};   \node (x) [below of=b] {$$};
  \node (c) [right of=b] {$\mathbf{r}g$};
  \node (a^b) [node distance=1cm,  left of=x] {$a\wedge b$};
  \node (y) [node distance=3cm, right of=a^b]  {$a^{g}$};
  \node (a) [node distance=2cm,  left of=b]{$a$};
  \node(z) [left of=a]{$\mathbf{d}f$};
  \node(w)[node distance = 3cm, left of=a^b]{$^{f}\!b$};
  \node(m)[node distance = 2.1cm, below of=a^b]{$e^{f}$};
  \node(e) [left of = w, node distance=1.5cm] {$e$};  
  \node (u) [node distance =2.8cm,left of =m] {$e\wedge ^{f}\!\!b$}; 
  \node (e^fg) [node distance = 2.8cm, right of = m]{$(e^{f})^{g}$};
  \draw [->](z)to node {$f$}(a);
  \draw[->] (a^b) to node [swap]{$_{a}\vert g$} (y);
  \draw[->] (b) to node {$g$} (c);
  \draw[red] (b) to node {$\leq_{L}$} (a^b);
  \draw [blue] (a) to node [swap]{$\leq_{R}$}(a^b);
  \draw[->] (w) to node[swap] { \;\;  $f\vert_{b}$} (a^b);
 \draw[->,bend right] (w) to node [swap]{$f\wedge g$} (y);
 \draw[blue] (z) to node [swap]{$\leq_{R}$}(w);
 \draw[red] (w) to node {$\leq_{L}$}(u);
 \draw[blue] (e) to node {}(u);
 \draw[->] (u) to node {} (m);
 \draw[->] (m) to node {} (e^fg);
 \draw[->,bend right](u) to node [swap]{$_{e}\vert (f\wedge g)$}(e^fg);
 \end{tikzpicture}
\end{figure}
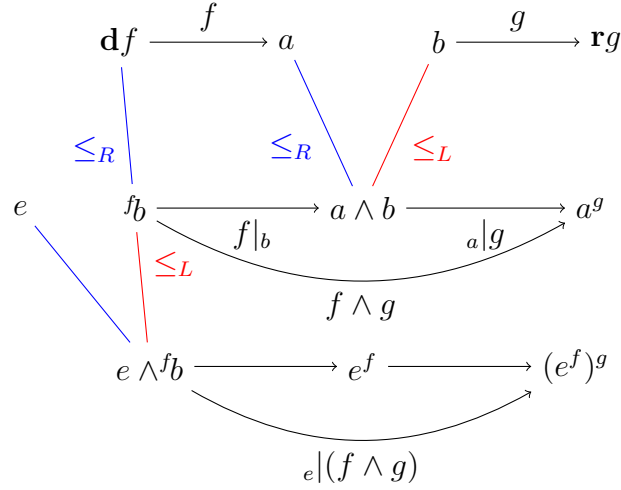
{It remains to prove associativity in full generality:} 
\begin{lemma}\label{assoc}
 For all $f,g,h\in \mathscr{G}$,
$%f\vee(g\vee h) = (f\vee g)\vee h $.  \text{   and     } 
f\wedge(g\wedge h) = (f\wedge g)\wedge h $.
\end{lemma}
\begin{proof}
 First we establish that when $f\circ g$ is defined, %$\mathbf{r}f = \mathbf{d}g$
 $(f\circ g)\wedge h = f\wedge (g\wedge h)$.  Let $r=\mathbf{r}g$ and $ d = \mathbf{d}h$; we have  
 \begin{eqnarray*}
 (f\circ g)\wedge h = &(f\circ g)\vert_{d}\;\circ\; _{r}\vert h &= (f\vert_{^{g}d} \circ g\vert_{d}) \circ\; _{r}\vert h \\
                              = &f\vert_{^{g}d} \circ (g\vert_{d} \circ\, _{r}\vert h ) &= f\vert_{^{g}d} \circ (g\wedge h );
\end{eqnarray*}
and since $^{g}d =\mathbf{d} (g\vert_{d}) = \mathbf{d}(g\wedge h)$, the latter is indeed $f\wedge (g\wedge h)$.  Now observe that, in the general case,    \[
%\begin{eqnarray*}
  (f\wedge g)\wedge h =(f\vert_{\mathbf{d}g}\;\circ \;_{\mathbf{r}f}\vert g)\wedge h = f\vert_{\mathbf{d}g}\wedge (\;_{\mathbf{r}f}\vert g\wedge h )    \]
  by the foregoing; and then, by Lemma \ref{1} (iii), we have \[
  f\vert_{\mathbf{d}g}\wedge (\;_{\mathbf{r}f}\vert g\wedge h )  = f\vert_{\mathbf{d}g}\wedge \;_{\mathbf{r}f}\vert( g\wedge h ) = f\wedge (g\wedge h ), 
  \]   completing the proof.
%  \\ &=  f\vert_{\mathbf{d}g}\wedge (\;_{\mathbf{r}f}\vert g\vert_{\mathbf{d}h} \circ\;_{(\mathbf{r}f)^{g}}\vert h ) &= 
% f\vert_{\mathbf{d}g}\wedge \;_{\mathbf{r}f}\vert (g\vert_{\mathbf{d}h} \circ\;_{\mathbf{r}g}\vert h )\\
% &= f \wedge (g \wedge h) .  
%\end{eqnarray*}
\end{proof}
\begin{lemma} \label{orth}
 $S  =  (\mathscr{G}, \wedge)$ is an orthodox semigroup.
\end{lemma}
\begin{proof}
 Lemma \ref{assoc} shows that $S$ is a semigroup. $S$ is regular, since $g\wedge g^{-1}\wedge g = g$ for any $g\in \mathscr{G}$.  
If $f\wedge f = f$, Fig. \ref{idem}
 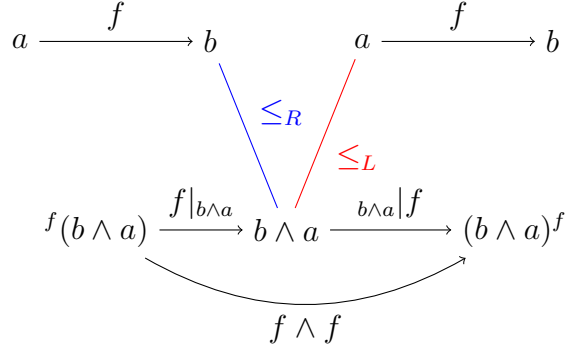
\begin{figure}[h!]
 \caption{Diagram for an idempotent}
 \label{idem}
\begin{tikzpicture}
\tikzset{node distance=2.5cm, auto}
  \node (b) {$a$};   \node (x) [below of=b] {$$};
  \node (c) [right of=b] {$b$};
  \node (a^b) [node distance=1cm,  left of=x] {$b\wedge a$};
  \node (y) [node distance=3cm, right of=a^b]  {$(b\wedge a)^{f}$};
  \draw[->] (a^b) to node {$_{b\wedge a}\vert f$} (y);
  \draw[->] (b) to node {$f$} (c);
  \draw[red] (b) to node {$\leq_{L}$} (a^b);
  %lefthand factor
  \node (a) [node distance=2cm,  left of=b]{$b$};
  \node(z) [left of=a]{$a$};
  \node(e)[left of=a^b]{$^{f}(b\wedge a)$};
  \draw [->](z)to node {$f$}(a);
  \draw [blue] (a) to node {$\leq_{R}$}(a^b);
  \draw[->] (e) to node {$f\vert_{b\wedge a}$} (a^b);
  \draw[->, bend right] (e) to node[swap] {$f\wedge f$} (y);
  %\draw[->] (I) to node [swap] {$q$} (B);
  %\draw[->] (I) to node {$p$} (A);
  %\draw[->] (B.120) to node [swap] {$n$} (X.240);
\end{tikzpicture}
\end{figure}
leads to $a = \;^{f}(b\wedge a), ~b = (b\wedge a)^{f}$ and so $b\wedge a = a^{f} = b$. Dually, $b\wedge a = \,^{f}b = a$; then $f\circ f = f$ and $f = {\bf i}_{a}$.  Thus $E(S) = B$ and $S$ is orthodox.  
\end{proof}

It also follows that every \ip~ is of the form $f\wedge f^{\ast}$. With $s\in S$, put $s^{+} = s\wedge s^{*}$ and $s^{-} = s^{*}\wedge s$.  Clearly $(s^{*})^{+} = s^{-}$ and $(s^{*})^{-} = s^{+}$, while $s^{+}\,\mathscr{R} \, s\, \mathscr{L}\, s^{-}$ ($\mathscr{R}$ and $\mathscr{L}$ being the usual Green's relations in $S$) and $s^{*}$ is the unique inverse of $s$ such that $s^{+}\,\mathscr{L} \, s^{*}\, \mathscr{R}\, s^{-}$.

\begin{theorem} \label{skehr}
 For all $s,t\in S$, there hold:
 
\begin{enumerate}[(i)]
 \item $s^{+}\wedge s^{+} = s^{+} = (s^{+})^{+} = (s^{+})^{-}$ \hspace{6pt} and \hspace{6pt} $s^{-}\wedge s^{-} = s^{-} = (s^{-})^{-} = (s^{-})^{+}$;
 \item $s^{+}\wedge s = s = s\wedge s^{-}$;
 \item $s\wedge s = s$ implies $s = s^{+} = s^{-}$;

\item $(s\wedge t)^{+} = (s\wedge t^{+})^{+}$ \hspace{4pt} and \hspace{4pt} $(s\wedge t)^{-} = (s^{-}\wedge t)^{-}$ ;
\item $(s^{+}\!\wedge t)^{+} = s^{+}\!\wedge t^{+}$  \hspace{4pt} and \hspace{4pt} $(s\!\wedge t^{-})^{-} = s^{-}\wedge t^{-}$.
\end{enumerate}
\end{theorem}
%(and so on). Then from the skew lattice axiom for $B$ we have the identities $f\vee (f^{\ast}f\wedge g^{\ast}g) = f, $  and three more by the dualities. 
 
\begin{proof}
 Parts (i)--(iii) follow by easy computation from the definitions and Lemma \ref{orth}. The definition of the extended $\wedge$ in the new notation (see Fig. \ref{+-}) reads $s\wedge t = (s\wedge t^{+})\circ (s^{-}\!\wedge t)$, and (iv) follows immediately. Part(v) is a consequence of (iv) with, respectively,  $s^{+}$ for $s$ and $t^{-}$ for $t$.   
\end{proof}

\begin{figure}[h!]
\caption{The pseudoproduct in $+ / - $ notation}
\label{+-}

\begin{tikzpicture}
\tikzset{node distance=2.5cm, auto}
  \node (b) {$t^{+}$};   \node (x) [below of=b] {$$};
  \node (c) [right of=b] {$t^{+}$};
  \node (a^b) [node distance=1cm,  left of=x] {$s^{-}\!\wedge t^{+}$};
  \node (y) [node distance=3.5cm, right of=a^b]  {$(s^{-}\!\wedge t)^{-}$};%
  \draw[->] (a^b) to node {$s^{-}\!\wedge t$} (y);
  \draw[->] (b) to node {$t$} (c);
  \draw[red] (b) to node {$\leq_{L}$} (a^b);
  %lefthand factor
  \node (a) [node distance=2cm,  left of=b]{$s^{-}$};
  \node(z) [left of=a]{$s^{+}$};
  \node(e)[node distance=3.5cm, left of=a^b]{$(s\wedge t^{+})^{+}$};%
  \draw [->](z)to node {$s$}(a);
  \draw [blue] (a) to node [swap]{$\leq_{R}$}(a^b);
  \draw[->] (e) to node {$s\wedge t^{+}$} (a^b);
  \draw[->, bend right] (e) to node {$s\wedge t$} (y);
  %\draw[->] (I) to node [swap] {$q$} (B);
  %\draw[->] (I) to node {$p$} (A);
  %\draw[->] (B.120) to node [swap] {$n$} (X.240);
\end{tikzpicture}

\end{figure}
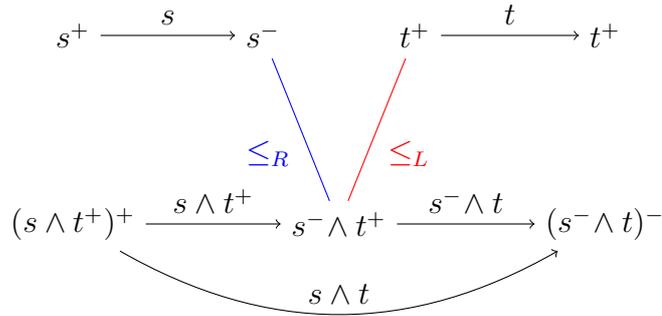

%\begin{
{\bf Remarks.   }  Theorem \ref{skehr} sets out the object part of a functor imitating that of the ESN theorem. There may be another occasion to describe the morphism part, which should also involve examining the properties in the Theorem,  since they include some of those forming the definitions of  restriction and Ehresmann semigroups.     
In fact, $(s\wedge t)^{+}\wedge s = s\wedge t^{+}$ \hspace{4pt} and \hspace{4pt} $t\wedge (s\wedge t)^{-} = s^{-}\!\wedge t$\; hold in a restriction semigroup as defined by Kudryavtseva \cite{Ku},  but fail here unless $B$ is a semilattice (in which case $S$ is inverse). The restriction and Ehresmann  classes are surveyed in \cite{Go}, and one may see the directions in which the ideas have been taken more recently in \cite{Jo} and \cite{Ku}.  This strand of research emphasises commuting idempotents, which distinguishes them from the present paper, where an element may have multiple left and right identities.  This may appear a little strange, but is the price to be paid for dealing with \emph{all} idempotents, not just a special subset.  More general contexts have already been considered, as in \cite{GW, Sz13, Sz14}, but the approach in hand is a natural and minimal extension of the inductive groupoid case, and returns to the spirit of groupoids as dealt with in another landmark paper---Lawson's \cite{La91}. 
Above all, our ultimate intent is to have $B$ as a skew lattice, and we deal with this in the next section.

We use the results from above:  beginning with  a skew lattice $B = (B,\wedge, \vee)$, we dualise the whole process of Section \ref{grpds1} to extend the join operation $\vee$ to $\mathscr{G}$, resulting in an algebra $S = (\mathscr{G}, \vee, \wedge)$.  

%Now we return to the linking conditions: we want a condition that says $(a\wedge f)\wedge b = a\wedge (f\wedge b)$, and analogously for the join operation.  These could be written out in the original restriction language, but we leave it in the above form for convenience. 
\section{Skew lattices of objects} \label{grpds2}
 Let $\mathscr{G}$ be a groupoid with composition $\circ$ and $B$ its set of objects, endowed with associative operations $\vee$ and $\wedge$ satisfying the absorptive axioms 
\[a\vee (a\wedge b) = a = a\wedge (a\vee b)~,\hspace {3em} (a\wedge b)\vee b  = b =  (a\vee b)\wedge b\]
 for a skew lattice  \cite{Le89, Le96}.  Then both $(B,\vee) $  and $(B, \wedge) $ are bands. %}, that is, idempotent semigroups.  
 Moreover each has a pair of natural preorders: in the \emph{lower} band $(B,\wedge) $ we write (continuing on from the preceding Section \ref{grpds1}) 
 \begin{itemize}
 \item $a\leq_{L} b$ \ioi $a = a\wedge b$, and  $a\leq_{R} b$ \ioi $a = b\wedge a$, 
\end{itemize}
and additionally  in the \emph{upper} band $(B,\vee) $ we write 
\begin {itemize}
\item $a\geq_{L} b$ \ioi $a = a\vee b$, and $a\geq_{R} b$ \ioi $a = b\vee a$. 
\end{itemize}  
We do not at this stage admit the usual convention that $\leq$ and $\geq$ are converse relations! 
The skew lattice absorptive axioms imply that $ a=a\vee b \iff a\wedge b =b \text{   and   } a\vee b = b \iff a = a\wedge b$, so that $a \leq_{L} b$ \ioi  $a=a\wedge b$ \ioi $b \geq_{R} a$; which is to say that $ \leq_{L}$ and $\geq_{R}$ are converse relations, as also 
$ \leq_{R}$ and $\geq_{L}$.  We write the relations in the form most suitable to the occasion. %Thus we need use only two of these four relations, which we mostly choose to be $\leq_{L}$ and $\leq_{R}$.

As a vertical dual to the set-up in Section \ref{grpds1}, we postulate \emph{left} and \emph{right extension} operations denoted 
$^{a}|,\, |^{a} :  \mathscr{G} \rightarrow \mathscr{G} $ such that
\begin{itemize}
\item  $^{a}|g$ is defined whenever $a \geq_{L} {\bf{d}}g $,  and %${\bf r}(_{a}| g)$ is written $a^{g}$ (so there are two actions here) 
%and \item 
$^{a}|g : a \rightarrow {\bf r}(^{a}| g)\geq_{L}{\bf r}g$;   
\end{itemize}

and again (lateral-) dually, 
\begin{itemize}

\item $g|^{a}$ is defined whenever $a \geq_{R} {\bf{r}}g $, with $g|^{a} : {\bf d}(g|^{a}) \rightarrow a$, \: ${\bf d}(g|^{a}) \geq_{L}{\bf d}g$.
\end{itemize} 
\begin{figure}[h!]
\caption{Left and right extension operators}
\label{extens}
%\subfloat{ 
% \begin{tikzpicture}    %top left
%\tikzset{node distance=2.5cm, auto}
%  \node (c) {$c$};   \node (x) [below of=c] {$$};
%  \node (d) [right of=c] {$d$};
%  \draw[->] (c) to node {$g$} (d);
%  \node (a) [below of=c] {$a$};
%  \node (y) [right of=a]  {${\bf r}(_{a}|g)$}; %a^{g}
%  \draw[->] (a) to node [swap]{$_{a}|g$} (y);
%  %\node (P1) [node distance=1.4cm, left of=P, above of=P] {$\hat{P}$};
% % \draw[->, bend left] (A.30) to node {$m$} (X.160);
%  \draw[red] (c.270) to node [swap]{$a\leq_{L}c$} (a.90);
%  %\draw[->,bend right] (B.60) to node [swap] {$n'$} (X.300);
%  %\draw[->] (I) to node [swap] {$q$} (B);
%  %\draw[->] (I) to node {$p$} (A);
%  %\draw[->] (B.120) to node [swap] {$n$} (X.240);
%\end{tikzpicture}
%}
%\hspace{2cm}
%\subfloat{   %topright
%\begin{tikzpicture}
%\tikzset{node distance=2.5cm, auto}
%  \node (c) {$c$};   \node (x) [below of=c] {$$};
%  \node (d) [right of=c] {$d$};
%  \draw[->] (c) to node {$g$} (d);
%  \node (a) [below of=c] {${\bf d}(g|_{a})$};
%  \node (y) [right of=a]  {$a$}; %a^{g}
%  \draw[->] (a) to node [swap]{$g\vert_{a}$} (y);
%  %\node (P1) [node distance=1.4cm, left of=P, above of=P] {$\hat{P}$};
% % \draw[->, bend left] (A.30) to node {$m$} (X.160);
%  \draw[blue] (d.270) to node {$a \leq_{R}d$} (y.90);
%  %\draw[->,bend right] (B.60) to node [swap] {$n'$} (X.300);
%  %\draw[->] (I) to node [swap] {$q$} (B);
%  %\draw[->] (I) to node {$p$} (A);
%  %\draw[->] (B.120) to node [swap] {$n$} (X.240);
%\end{tikzpicture}
%%\end{center}
%}
\vspace{3mm}

%\begin{center}
%\begin{subfigures}
\subfloat{ 
 \begin{tikzpicture}	%bottom left
\tikzset{node distance=2.5cm, auto}
  \node (c) {$a$};   
  \node (x) [below of=c] {$c$};
  \node (d) [right of=c] {${\bf r}(^{a}\vert g)$};
  \draw[->] (c) to node {$^{a}\vert g$} (d);
  \node (a) [below of=c] {$c$};
  \node (y) [right of=a]  {$d$}; %a^{g}
  \draw[->] (a) to node [swap]{$g$} (y);
  %\node (P1) [node distance=1.4cm, left of=P, above of=P] {$\hat{P}$};
 % \draw[->, bend left] (A.30) to node {$m$} (X.160);
  \draw[blue] (c.270) to node [swap]{$a\geq_{L}c$} (a.90);
  %\draw[->,bend right] (B.60) to node [swap] {$n'$} (X.300);
  %\draw[->] (I) to node [swap] {$q$} (B);
  %\draw[->] (I) to node {$p$} (A);
  %\draw[->] (B.120) to node [swap] {$n$} (X.240);
\end{tikzpicture}
}
\hspace{2cm}
%\end{center}
\subfloat{
%\begin{center}  
 \begin{tikzpicture}	%bottom right
 \tikzset{node distance=2.5cm, auto}
  \node (c) {${\bf d}(^{a}|g)$};
  \node (d) [right of=c] {$a$};
  \draw[->] (c) to node {$g\vert^{a}$} (d);
  \node (a) [below of=c] {$c$};
  \node (y) [right of=a]  {$d$}; %a^{g}
  \draw[->] (a) to node [swap]{$g$} (y);
  %\node (P1) [node distance=1.4cm, left of=P, above of=P] {$\hat{P}$};
 % \draw[->, bend left] (A.30) to node {$m$} (X.160);
  \draw[red] (d.270) to node {$a\geq_{R}d$} (y.90);
  %\draw[->,bend right] (B.60) to node [swap] {$n'$} (X.300);
  %\draw[->] (I) to node [swap] {$q$} (B);
  %\draw[->] (I) to node {$p$} (A);
  %\draw[->] (B.120) to node [swap] {$n$} (X.240);
 \end{tikzpicture} 
 }
\end{figure}

The relevant diagrams appear in  Figure \ref{extens}.  Again we are able to write $a\vee g$ for $^{a}\vert g$ and by extension for $^{a\vee \mathbf{d}g}\vert g$, and $a_{g}$ for ${\bf r}(^{a}| g)$; similarly,  $_{g}a = {\bf d}(^{a}|g)$.

The postulates vertically dual to those of the preceding Section \ref{grpds1} are to hold also, and we list them here, using the abbreviated notation developed in Section \ref{grpds1} and without further explanation; moreover we only give one-sided forms, assuming the lateral duals hold by implication.  Thus each postulate stands for a quartet (although some are self-dual or may have vertical and lateral duals equivalent). %{\it Modify to vee and Add other postulates \dots}
\begin{enumerate}[(i)]	%[label = (\roman*)] %[label=(\roman*)]
 \item (identities) \hspace{1em}% or idempotency 
 ${\bf d}g\vee g = g$;
 \item (preorders) \hspace{1em}if $a \geq _{R}b$, then $ {a}\vee{\bf i}_{b} = {\bf i}_{a\vee b}$;
 \item (transitivity) \hspace{1em} if $a \geq_{R} b \geq_{R}{\bf{d}}g $,
 then ${a}\vee g = (a\vee b)\vee g = a\vee (b\vee g)$;  
  \item (composition) \hspace{1em} if $f\circ g$ is defined (so that ${\bf r}\!f = {\bf d}g$), then 
 ${a}\vee(f\circ g) = (a \vee f)\circ (a_{f}\vee g)$;
 \item (dual of Theorem \ref{orth})\hspace{1em}  $(a\vee f)\vee b = a\vee (f\vee b) ,$ for all $a,b\in B$ and $f\in \mathscr{G}.$ 
\end{enumerate}

The vertical dual of the development in Section \ref{grpds1}  extends the join operation $\vee$ to all of $\mathscr{G}$ and of course the dual results hold.  In particular, we note that $a\vee f\vee f^{\ast} =  a\vee f\vee (a\vee f)^{\ast} $.  
Moreover, extra postulates are required to establish compatibility conditions between the restriction and extension operators which reflect the skew lattice character of $B$.  %(It would be interesting to consider what might arise in the broadly similar context of \emph{double bands} \cite{db} linked by the exchange identities $(a\vee b)\wedge (c\vee d) = (a\vee c)\wedge (b\vee d)$.)  

%If the above must be added as an axiom (and does not follow from the other conditions), then the same must be done for associativity in general. Again this is quite opaque in terms of the restrictions, extensions and composition operations in the groupoid, so for now we simply add axioms which say 
%{\it Yet to prove:} \[f\vee(g\vee h) = (f\vee g)\vee h \text{   and     } f\wedge(g\wedge h) = (f\wedge g)\wedge h\]
%---compare the above $a\wedge (f\wedge b)$ etc. as a special case.  (Writing it out in terms of more primitive elements is possible but not enlightening.)%of the general $\wedge$ and $\vee$ operations.

Observe that when $f\circ g$ is defined, $f\vee g =f\wedge g = f\circ g$; in particular, $f\vee f^{-1} =f\wedge f^{-1} = {\bf d}f$, etc.   We write (to conform to precedent) $f^{\ast}$ in place of $f^{-1}$, and may as well write $ff^{\ast}$ for $f\circ f^{\ast}= f\wedge f^{\ast} = f\vee f^{\ast}$, etc.  The identification of $\mathbf{i}_{a}$ with $a$ also identifies $a^{\ast}$ with $\mathbf{i}_{a}^{-1} = \mathbf{i}_{a}$ and so $(a\wedge f)^{\ast} $ with ${a^{f}}\!\wedge f^{\ast}$, and similarly $(a\vee f)^{\ast} = a^{f} \vee f^{\ast}$.  

The restriction and extension operators should also be linked through the skew lattice orders.  Consider  any object $a \in B$ and morphism $f$; write ${\bf d}f = d = f f^{\ast}$ and ${\bf r}f = r = f^{\ast}f$, and set  $ b = r\vee a \, \geq_{R}\, r$.  % for some $b\in B$. 
 Then $\,_{a}\vert f \colon a \rightarrow a^{f}$ exists, and $a^{f}\leq_{L} r$, which is to say $r \geq_{R} a^{f}$, and so there is $(_{a}\vert f)\vert^{r} \colon d' \rightarrow r$.  When $f = \mathbf{i}_{d}$, we see that $d' = d$ so it is reasonable that this hold in general. See Fig. \ref{link}. %\begin{center}
 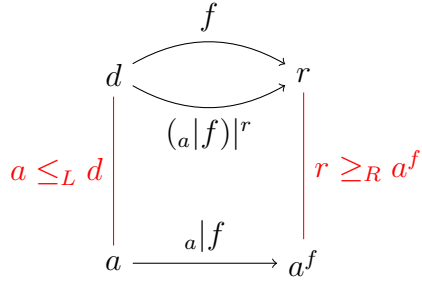
\begin{figure}[h!]
\caption{Restriction and extension operators linked}
\label{link}
 \begin{tikzpicture}
 \tikzset{node distance=2.5cm, auto}
  \node(d){$d$};
  \node(a)[below of=d] {$a$};
  \node(r)[right of=d] {$r$};
  \node(a^{f})[below of = r]{$a^{f}$};
 \draw[->, bend left] (d) to node {$f$} (r);
  \draw[red](d) to node [swap]{$a\leq_{L}d$} (a);
  \draw[->](a) to node{$_{a}\vert f $} (a^{f});
 \draw [red] (a^{f}) to node[swap] {$r\geq_{R}a^{f}$} (r);
 \draw [->, bend right] (d) to node [swap]{$(_{a}\vert f)\vert^{r} $} (r); %  \draw[->, bend right](c) to node[swap]{$_{(=f\vee a)}$} (b);  
\end{tikzpicture}
\end{figure}

Indeed we shall require, as a linking condition, that $(_{a}\vert f)\vert ^{r} = f$ and so we add to the previous list the axiom 
\begin{enumerate}[(vi)]	%[label = (\roman*)] %[label=(\roman*)]
 \item 
$f = (a\wedge f)\vee f^{\ast}f ,  \text{      or equivalently      } f f^{\ast}= (a\wedge f)\vee f^{\ast}.$
\end{enumerate}  
 We also assume the lateral and order duals, which  are interpreted similarly. Note that when $f\in B,  ~ f = f^{\ast} = ff^{\ast}$, and this equation reduces to the absorptive identity \[f = (a\wedge f)\vee f\]  of skew lattices.  
 
 We shall (tentatively) refer to a groupoid satisfying these conditions as a \emph{skew inductive groupoid}.
 Theorem \ref{orth} applies and assures the existence of an algebra $(S,\vee, \wedge, \ast)$ arising from a skew inductive groupoid. We now seek to characterise such an algebra axiomatically.
 
 %This suggests  axiom (v) below.
%\end{tikzpicture}			f\vert^{b}
%\end{center}
%
%Since $B$ is a skew lattice, $d\leq_{L} c$  and we may consider $_{d}\vert g$. It is reasonable that this recovers $f$, thus that $f = _{d}\!\vert g =_{d}\!\vert (f\vert^{b})$. 
%Remembering that  $b = g^{\ast}\vee g$, that $d = f\wedge f^{\ast}$ and using the alternative notation, we have the identity $f = {f\wedge f^{\ast} }\wedge (f\vee {g^{\ast}}\vee g)$.  There are three further analogous identities (from the horizonal and vertical dualities).  Notice that if $f,g \in B$, all these reduce to the skew lattice identities! 
%$t = _{d}\vert s\vert^{t\vee t^{\ast}}$

\section{Algebraic characterisation}\label{ax}

%From the previous section \ref{grpds2}  we assemble axioms for skew inductive groupoids.  

Let $(S,\vee,\wedge,*)$ be an algebra of signature $(2,2,1)$, with $\vee, \wedge \colon S\times S \rightarrow S$ and $^{\ast}: S \rightarrow S$, that satisfy, for all $s,t\in S$:

\begin{enumerate}[(i)]
 \item $(S, \vee)$ and $(S,\wedge)$ are associative (thus, semigroups); 
 \item $(s^{\ast})^{\ast} = s$;
 \item $s\vee s^{\ast} = s\wedge s^{\ast} = (s\wedge s^{\ast})^{\ast}$;
 \item $s\vee s^{\ast}\vee s = s = s\wedge s^{\ast}\wedge s$; 
 \item $s^{2}=s$ implies $s=s^{\ast}$;
% \item $s\vee( s^{\ast}\wedge s\wedge t\wedge t^{\ast}) = s = s\wedge (s^{\ast}\vee s\vee t \vee t^{\ast});$
% \item $(s^{\ast}\wedge s\wedge t\wedge t^{\ast})\vee t = t =(s^{\ast}\vee s\vee t\vee t^{\ast})\wedge t;$ 
 \item ${s\vee s^{\ast} }\vee (s\wedge {t^{\ast}}\wedge t) = s = {s\wedge s^{\ast} }\wedge (s\vee {t^{\ast}}\vee t)$ and lateral duals;
 \item $ s\vee s^{\ast}\vee t\vee t^{\ast} =  s\vee s^{\ast}\vee t\vee (s\vee s^{\ast}\vee t)^{\ast} $ and duals;
 \item $s^{\ast}\vee s = t\vee t^{\ast}$ implies 
 \[ s\vee t\vee(s\vee t)^{\ast} = s\vee s^{\ast} ~ \text{and}~ (s\vee t)^{\ast}\vee s\vee t = t^{\ast}\vee t~\text{and}  \]
\[ s\wedge t\wedge(s\wedge t)^{\ast} = s\wedge s^{\ast} ~ \text{and}~ (s\wedge t)^{\ast}\wedge s\wedge t = t^{\ast}\wedge t .\]

\end{enumerate}

%In fact, (v) and (vi) (which are laterally dual) are equivalent to  skew lattice identities for $E(S)$ in this context.  In turn those follow from (vii) and its lateral and order duals.   (So (v) and (vi) may be replaced by any axioms defining a skew lattice on $E(S)$.) 

% and it is worthwhile examining (vii) in greater detail. Consider $f$ with domain ${\bf d}f = d = f\vee f^{\ast}$  and range  $ {\bf r}f = r$, and suppose $ b \geq_{R} r$ for some $b\in E(S)$.  Then $f\vert^{b} = g = f\vee b$ exists, say $g \colon c\rightarrow b$, as in the diagram, and $c\geq_{R}d$:
 
%\begin{center}
% \begin{tikzpicture}
% \tikzset{node distance=2.5cm, auto}
%  \node(c){$c$};
%  \node(d)[below of=c] {$d$};
%  \node(r)[right of=d] {$r$};
%  \node(b)[above of = r]{$b$};
% \draw[->, bend right] (d) to node [swap]{$f$} (r);
%  \draw[red](b) to node {$b\geq_{R}r$} (r);
%  \draw[->](c) to node{$g$} (b);
%  \draw[->](c) to node[swap]{$_{(=f\vee b)}$} (b);  
%  \draw [red] (c) to node[swap] {$c\geq_{R}d$} (d);
%  \draw [->, bend left] (d) to node []{$_{d}\vert g$} (r);
%\end{tikzpicture}
%\end{center}
%
%But this means that $d\leq_{L}c$, and so $_{d}\vert g$ exists; axiom (vii) insists that $_{d}\vert g = f$.  The lateral and order duals are interpreted similarly. 

The properties in Section \ref{grpds2}, including Lemmas \ref{1} and \ref{assoc} and Theorem \ref{orth}, show that we were able to construct such an object from a skew inductive groupoid.  Conversely, we have 
\begin{theorem}\label{main}
Let $(S,\vee,\wedge, {\ast})$ satisfy axioms (i)--(viii), and form a small category $\mathscr{C}$ as follows.
\begin{itemize}
 \item $\emph{Ob}(\mathscr{C}) = \{s\vee s^{\ast}\colon s\in S\}$,
 \item  $\emph{Mor}(\mathscr{C}) = \{\widehat{s} = (s\vee s^{\ast}\!,~s,~s^{\ast}\vee s) \colon s\in S\}$,
 \item  when $s^{\ast}\vee s = {\bf r}(\widehat{s}) = {\bf d}(\widehat{t}) = t\vee t^{\ast}, ~~\widehat{s}\circ \widehat{t}$ is defined and $ \widehat{s}\circ \widehat{t} = (s\vee s^{\ast},~st,~t^{\ast}\vee t).$
\end{itemize}
Then $\mathscr{C}$ is a skew inductive groupoid whose pseudoproduct gives an orthodox semigroup isomorphic with $S$. 
\end{theorem}
\begin{proof}
From (vi) we have that $\emph{Ob}(\mathscr{C})$ is a skew lattice. 
Clearly composition when defined for triples is associative, and each $(s\vee s^{\ast}\!,~s\vee s^{\ast}\!,~s^{\ast}\vee s)$ is the identity at object $s\vee s^{\ast}$. 
Morphism $\widehat{s} = (s\vee s^{\ast}\!,~s,~s^{\ast}\vee s)$ has inverse 
$\widehat{s}^{-1} = (s^{\ast}\vee s, s^{\ast}, s\vee s^{\ast})$.   
The restriction and extension operators must be defined: for a morphism $\widehat{s} = (s\vee s^{\ast},~s,~s^{\ast}\vee s)$ and an object $a$ such that $a\geq_{L}s\vee s^{\ast}$ (i.e. $a = a\vee s\vee s^{\ast}$), we set 
\[^{a}\vert \widehat{s} = (a,~a\vee s,~(a\vee s)^{\ast}\vee a\vee s) .\]  %>>>>>>>>stuff to be finished here<<<<<<<<<
The r.h.s. is indeed in $\emph{Mor}(\mathscr{C}) $: by equation (\ref{agg*}),  %$({a}\wedge f)^{\ast} $ with ${a^{f}}\!\wedge (f^{\ast})$.
$ (a\vee s) \vee (a\vee s)^{\ast} = (a\vee s) \vee s^{\ast} = a$ by hypothesis.  Moreover, 
\[ \mathbf{r}(^{a}\vert \widehat{s}) \vee s \vee s^{\ast} = ((a\vee s)^{\ast}\vee a\vee s) \vee s \vee s^{\ast} =\mathbf{r}(^{a}\vert \widehat{s}),\] so $\mathbf{r}(^{a}\vert \widehat{s}) \geq_{L}\mathbf{r}(s)$, as required for an extension operator.

Next, the postulates of Section \ref{grpds2} have to be verified.  It is useful to observe that the right [left] component of a left- [right-]extended morphism depends solely on the middle component, and so may safely be left unspecified (written $\sim$) in certain calculations.
\begin{enumerate}[(i)]
\item (``identity'') follows from regularity (axiom (iv)).
\item (``preorder'') Assume $  a= b\vee a$.  By definition, \[
^{a}\vert \mathbf{i}_{b} =  ^{(a,a,a)}\!\vert (b,b,b) = (a\vee b, a\vee b, (a\vee b)^{\ast}\vee a\vee b) = \mathbf{i}_{a \vee b} .
\]
\item (``transitivity'') First, $ ^{b}\vert \widehat{s} = \;^{(b,b,b)}\vert (s\vee s^{\ast} ,s,s^{\ast}\vee s) = (b\vee  s\vee s^{\ast},b\vee s,\sim) $, so \newline
$ ^{a}\vert (^{b}\vert \widehat{s} ) = \;^{a}\vert (b\vee  s\vee s^{\ast},b\vee s,\sim) = (a\vee b\vee  s\vee s^{\ast},a\vee b\vee s,\sim) = \;
^{(a\vee b)}\vert \widehat{s}$,~by associativity of $S$.
\item (``composition'')
\begin {align*}
(^{a}\vert\widehat{s})\vert^{b} &= (a\vee s\vee s^{\ast}, a\vee s, (a\vee s)^{\ast}\vee a\vee s)\vert^{b} \\&= ((a\vee s\vee b)(a\vee s\vee b)^{\ast}, a\vee s\vee b, (a\vee s)^{\ast}\vee a\vee s\vee b),
\end{align*}
while
\begin {align*}
^{a}\vert(\widehat{s}\,\vert^{b}) &= \;^{a}\vert((s\vee b)(s\vee b)^{\ast}, s\vee b, (s\vee s)^{\ast}\vee b) \\
&= (a\vee (s\vee b)( s\vee b)^{\ast}, a\vee s\vee b, (a\vee s\vee b)^{\ast}\vee a\vee s\vee b),
\end{align*}
and by (the lateral dual of) axiom (vii), these are equal.
\item (``dual of Theorem \ref{orth}'')\hspace{1em} This follows from associativity in $S$.
\end{enumerate}
 In this manner we have constructed a groupoid $\widehat{S}$ over a skew lattice of objects.  Now suppose that $S$ arises from the original groupoid $\mathscr{G}$. The mapping  $\mathscr{G} \longrightarrow \widehat{S}$ given by 
 $g \mapsto (\mathbf{d}g, g, \mathbf{r}g)$ is routinely an isomorphism, simply representing different ways of describing $\mathscr{G}$; the fact that it factors through $S$ completes the proof.  
\end{proof}

\section{models}\label{gb}
Do such objects even exist?  One special case occurs with $\mathscr{G}$ a true inductive groupoid and $B$ a lattice.  Such a combination gives rise to two inverse semigroups (monoids in fact), and an easy way to realise such an object is by taking the direct product of a group with a lattice.  This could be the inspiration for a less trivial example, as follows.

%%%%%%%%%%%%%
Let a group $G$ act by automorphisms on a band $B$.  Then we may consider the semidirect product $S=G\ltimes B$ with base set $G\times B$ and multiplication, for $u,v \in G$ and $a,b \in B$, \[(u,a)(v,b)=(uv, a^{v}\cdot b).\]

This situation was studied some time ago by Miklos Hartmann and Maria Szendrei \cite{Ha, HSz} and maybe others I have not yet found; iand it seems to have been generalised in \cite{GW}.  All we need to note here is that 
\begin{itemize}
\item idempotents are exactly the elements $(1,a)$, and $E(S) \cong B$,
\item $S$ is regular with an involution $(u, a)^{*}=(u^{-1},a^{u^{-1}})$, such that
\item $(u,a)(u,a)^{*} =  (u,a)(u^{-1},a^{u^{-1}}) = (1, a^{u^{-1}}), \hspace{1em} (u,a)^{*} (u,a) = (1,a) $
\item $(u,a)(u,a)^{*}(u,a) = (u,a)$
\item so $S$ is orthodox but not inverse
\item and $*$ is {\bf not} an anti-automorphism.
\end{itemize}

On this last point, let us observe that $[(u,a)(v,b)]^{\ast} = (v^{-1}u^{-1}, a^{u^{-1}}\wedge b^{v^{-1}u^{-1}})$, so \[
[(u,a)(v,b)]^{\ast}[(u,a)(v,b)] = (1, (a\wedge b^{v^{-1}})^{u^{-1}}),	\]
which reduces to $(1,a^{u^{-1}})$ precisely when $a = b^{v^{-1}}$, i.e., when $(u,a)^{\ast}(u,a) = (v,b)(v,b)^{\ast}$.
In structural terms, these are both equivalent to $(u,a)\,\mathscr{R}\,(u,a)\!(v,b)\,\mathscr{L}\,(v,b)$.  (This may also be relevant to criteria for composibility in the double-orthodox semigroup set-up.)

We may conventionally write a ``normal form'' $ua $ for $(u, a)$.  Then $(u,a)=(u,\top)(1,a)$ when $B$ has a top element $\top$,  and so $S= GB$ and we have the factorisable case.  Otherwise, $S\cup G$ is factorisable and $S$ almost factorisable.  See also Rina-e Zenab's recent article \cite{Ze}, and its references, for Zappa-Sz\'{e}p products of which this is also an example.

The map $\phi :S \rightarrow G, \hspace{6pt}ua \mapsto u$ %$ defined by $
partitions $S$ into blocks $S_{u}= u\phi ^{-1}$, and $S_{u}$ is isomorphic with $B$ when given the sandwich multiplication (for $ua, ub \in S_{u})$,   $ ua \star ub = ua(u^{-1})ub =uab$; so  $S$ is a ``group of (isomorphic) sandwich bands''. Conversely, given such a $\{B_{u}\colon u \in G\}$ with connecting isomorphisms 
\[\{\lambda_{u,v}, \rho_{u,v}: B_{u}\rightarrow B_{v}\}\] satisfying the right axioms, one may reconstruct $S=\bigcup B_{u}$ with multiplication (for $s\in B_{u}, t\in B_{v}$) given by
 \[s\cdot t = s\rho_{u,uv} \star t\lambda_{v,uv}\] 
 with $\star$ the multiplication in $B_{uv}$.  (There is nothing special about this, it's just another description of a semidirect product.)

Then we can see what happens when we do it twice over, replacing $\cdot$ by $\wedge$ and $\vee$ .  
(We will end up with an algebra of signature $(2,2,1)$.) Note that $(u,a) = u\wedge a = u\vee a$ in the normal form, and 
\[ua\wedge (ua)^{*} = ua\wedge u^{-1}a^{u^{-1}} = 1\wedge a^{u^{-1}}, \hspace{2em} (ua)^{*}\wedge ua = u^{-1}a^{u^{-1}}\wedge ua = 1\wedge a ;\]  and exactly the same with the $\vee$ operation.  Thus $s\vee s^{*} = s\wedge s^{*} = ss^{*}$, etc. (using juxtaposition where either main operation may be applied). So the absorptive identity 
$a\vee(a\wedge b) = a$ is equivalent to 
 \[ s^{*}s\vee (s^{*}s \wedge t^{*}t) = s^{*}s  \hspace {2em} \text{and so to} \hspace{2em} s\vee (s^{*}s \wedge t^{*}t) = s ;\]
and likewise for the lateral and order  duals.
 
The theory of inverse semigroups suggests that we investigate an idempotent-separating $*$-congruence $\sim $ of such a $G\ltimes  B$.  If  $ua \sim vb$ then $a \sim b = a$; so we are led to consider the subgroups 
$K_{a}:=\{u\in G | ua = a\} .$ Now  $$K_{a} \subseteq K_{a\vee b} \subseteq K_{(a\vee b)\wedge b} =K_{b} $$
for all $a,b \in B$; thus $K_{a} = K \unlhd G$, say; and we may as well have started with $G/K$.
%%%%%%%%%%%%

The groupoid version of $G\ltimes  B$ may be presented as follows.  Given a skew lattice $B$ and a group $G$ acting by automorphisms on $B$, make a category with objects from $B$ and morphisms $(b,g,b^{g})$. The composition $(b,g,b^{g})\circ (c, h, c^{h})$ is defined exactly when $b^{g}=c$, and is given by $(b,g,b^{g})\circ (c, h, c^{h}) = (b, gh, b^{gh})$.
If one works it through, one has the pseudoproduct  \[
(b,g,b^{g})\otimes_{\wedge} (c,h,c^{h}) = (b\wedge c^{g{-1}},gh, b^{gh}\wedge c^{h}) ,  \]
which we may abbreviate $(g,b^{g})\cdot (h,c^{h}) = (gh, b^{gh}\wedge c^{h})$,  the semidirect product.

\section{Further comments}
The restriction idea may provide another useful way of thinking about skew lattices.
It remains to describe categories  of  orthodox $\ast$-semigroups and skew inductive groupoids, and functors establishing an equivalence between them.  Refinement of the axioms may also be possible, and the relationships with the approach of actions (of objects on morphisms and morphisms on objects) should be explored.  The connexions with restriction and Ehresmann semigroups need to be teased out.  More ``natural'' or concrete examples would be desirable---for example, can they be found 
 in rings or override algebras?
%\item Is there a model of such $S$ with a logical flavour?
%\end{itemize}
\section*{Acknowledgements}
The author thanks Vicky Gould, Michael Kinyon and Anya Kudryavtseva for their interest and input in discussions, both in person and by email, which helped convert a half-baked workshop idea into an interesting project.

\end{document}